\documentclass[11pt]{article}

\usepackage{amsmath, amsthm, amssymb}
\usepackage{amsfonts,url,hyperref}
\usepackage{epsfig,tikz}
\usepackage{fullpage}

\newtheorem{theorem}{Theorem}[section]
\newtheorem{proposition}[theorem]{Proposition}
\newtheorem{lemma}[theorem]{Lemma}
\newtheorem{corollary}[theorem]{Corollary}

\newcounter{Examplecount}
\newenvironment{remark}[1][Remark.]{\begin{trivlist}
\item[\hskip \labelsep {\bfseries #1}]}{\end{trivlist}}

\newcommand{\done}[1]{}

\makeatletter
\def\Ddots{\mathinner{\mkern1mu\raise\p@
\vbox{\kern7\p@\hbox{.}}\mkern2mu
\raise4\p@\hbox{.}\mkern2mu\raise7\p@\hbox{.}\mkern1mu}}
\makeatother

\newcommand\beq{\begin{equation}}
\newcommand\eeq{\end{equation}}
\newcommand\bce{\begin{center}}
\newcommand\ece{\end{center}}
\newcommand\bea{\begin{eqnarray}}
\newcommand\eea{\end{eqnarray}}
\newcommand\ba{\begin{array}}
\newcommand\ea{\end{array}}
\newcommand\ben{\begin{enumerate}}
\newcommand\een{\end{enumerate}}
\newcommand\bit{\begin{itemize}}
\newcommand\eit{\end{itemize}}
\newcommand\brr{\begin{array}}
\newcommand\err{\end{array}}
\newcommand\bt{\begin{tabular}}
\newcommand\et{\end{tabular}}

\newcommand\ms{\medskip}

\renewcommand\S{{\mathcal S}}
\def\red{\operatorname{st}}
\def\O{O}
\def\I{\mathcal{I}}
\def\lext{\mathcal{L}}
\def\N{{\mathcal N}}

\title{The most and the least avoided consecutive patterns\thanks{Research
partially supported by NSF grant DMS-1001046}}
\author{Sergi Elizalde~\thanks{Department of Mathematics,
Dartmouth College, Hanover, NH 03755. E-mail: \texttt{sergi.elizalde@dartmouth.edu}}
}

\begin{document}

\maketitle

\begin{abstract}
We prove that the number of permutations avoiding an arbitrary consecutive pattern $\sigma$ of length $m$ is asymptotically largest when $\sigma=12\dots m$, and smallest when $\sigma=12\dots(m-2)m(m-1)$.
This settles a conjecture of the author and Noy from 2001, as well as another recent conjecture of Nakamura.
We also show that among non-overlapping patterns of length $m$, the pattern $134\dots m2$ is the one for which the number of permutations avoiding it is asymptotically largest.
\end{abstract}


\section{Introduction and background}

The notion of consecutive patterns is a variation of the more standard definition of patterns in permutations. 
In an occurrence of a consecutive pattern in a permutation, the positions of the entries are required to be adjacent.
Consecutive patterns appear naturally in fundamental combinatorics. For instance,
occurrences of $21$ are descents of the permutation, occurrences of $132$ and $231$ are peaks, and permutations avoiding $123$ and $321$ are called alternating permutations.
Other than these implicit appearances, the systematic study of consecutive patterns in permutations was started in 2001 by Elizalde and Noy~\cite{EliNoy},
who gave generating functions counting occurrences of some consecutive patterns in permutations, by expressing them as solutions of certain differential equations.
Since then, significant progress has been made by many authors, including
Aldred, Atkinson, Baxter, B\'ona, Claesson, Dotsenko, Duane, Ehrenborg, Jones, Khoroshkin, Kitaev, Mansour, McCaughan, Mendes, Nakamura, Perry, Remmel, Shapiro, and Zeilberger.
However, the main conjecture from the original paper~\cite{EliNoy} has remained open all these years.
The conjecture states that among all consecutive patterns of a fixed length $m$, the increasing pattern $12\dots m$ (and, by symmetry, the decreasing pattern $m\dots21$)
is the one for which the number of permutations avoiding it is asymptotically largest. This conjecture is often mentioned in the literature~\cite{Bon08,BonPP,Bon,Nak}.
The first main result of the present paper is a proof of this conjecture. We will refer to it as the {\em Consecutive Monotone Pattern (CMP) Conjecture}, given that $12\dots m$ and $m\dots 21$ are sometimes called {\em monotone} patterns.

Aside from supporting experimental evidence, the intuition behind the conjecture can be explained as follows.
It is easy to see that the total number of occurrences of a pattern $\sigma$ of length $m$ in all $n!$ permutations of length $n$ does not depend on $\sigma$. When $\sigma$ is monotone,
occurrences of $\sigma$ can overlap with each other in more ways than for any other pattern, so a lot of permutations of length $n$ will contain many occurrences of $\sigma$. It seems plausible then
that, to compensate, there must be a large number of permutations (more than for any other pattern) not containing any occurrence of $\sigma$.

Even so, the is a caveat in the above reasoning. A similar intuitive argument seems to suggest that the analogous conjecture for classical patterns (namely,
when entries in an occurrence are not required to be adjacent) should hold as well. However, this is known to be false: B\'ona~\cite{Bon97} showed that
for $n\ge7$, there are fewer permutations of length $n$ avoiding the classical pattern $1234$ than avoiding the classical pattern $1324$.

In a different attempt to shed some light on the conjecture for consecutive patterns,
B\'ona~\cite{Bon08} considered another notion of pattern containment that is even more restrictive, by requiring not only the positions but also the values of an occurrence of the pattern to be adjacent.
Under this restrictive definition, he was able to show that the analogue of the conjecture holds for {\em most} patterns, that is,
the number of permutations avoiding the pattern $12\dots m$ in adjacent values and positions is larger than for any most other patterns of length $m$ (see~\cite{Bon08} for details).

The CMP Conjecture is known to hold in some special cases.
The case $m=3$ was proved in~\cite{EliNoy}. More recently, Elizalde and Noy~\cite{EliNoy2} showed
that the number of permutations avoiding $12\dots m$ is asymptotically larger than the number of permutations avoiding any fixed non-overlapping pattern of length $m$.
Non-overlapping patterns are those for which two occurrences cannot overlap in more than one position.

The second main result of this paper is the proof of a recent related conjecture of Nakamura~\cite[Conjecture~2]{Nak}
which, made on computational evidence, states that the pattern $12\dots (m-2)m(m-1)$ is the one for which the number of permutations avoiding it is asymptotically {\em smallest}.
This conjecture is {\em complementary} to the CMP Conjecture.
We remark that, once again, the analogue for classical patterns of Nakamura's conjecture does not hold: as shown by B\'ona~\cite{Bon97},
the are more permutations of length $n\ge 6$ avoiding the classical pattern $1243$ than avoiding the classical pattern $1423$.
In fact, it was proved in~\cite{BWX} that, as classical patterns, the number of permutations avoiding $12\dots m$ is the same as the number of permutations avoiding $12\dots (m-2)m(m-1)$.
It is therefore surprising that their behavior is completely different as consecutive patterns, since in such setting these two are the most and the least avoided patterns, respectively.

The third result in this paper concerns non-overlapping patterns. We prove a recent conjecture of the author and Noy~\cite{EliNoy2}
stating that among non-overlapping patterns of length~$m$, the pattern $134\dots m2$ is the one for which the number of permutations avoiding it is asymptotically largest.

In the rest of this section we give some background on consecutive patterns and we set the notation for the rest of the paper.
We also describe some of the ingredients in our proofs: singularity analysis of generating functions, the cluster method of Goulden and Jackson, and linear extensions of posets.
The CMP conjecture is proved in Section~\ref{sec:most}. In Section~\ref{sec:nol} we discuss non-overlapping patterns, and we find the most and the least avoided ones.
Finally, in Section~\ref{sec:least} we prove Nakamura's conjecture stating that $12\dots (m-2)m(m-1)$ is the least avoided pattern of length $m$. We end discussing some open problems in Section~\ref{sec:final}.

\subsection{Consecutive patterns}

For a sequence $\tau=\tau_1\tau_2\dots\tau_k$ of distinct positive integers, let $\red({\tau})$ denote the permutation of length~$k$ obtained by
replacing the smallest entry of $\tau$ with~$1$, the second smallest with~$2$, and so on. For example, $\red(394176) = 263154$.
Given permutations $\pi \in \S_n$ and $\sigma \in \S_m$, an {\em occurrence} of $\sigma$ in $\pi$ as a consecutive pattern is a subsequence of $m$ adjacent entries of $\pi$ such that $\red(\pi_i \cdots \pi_{i+m-1}) = \sigma$.
For example, in $\pi=15243$, the subsequences $152$ and $243$ are two occurrences of the pattern $\sigma=132$.
Denote by  $c_\sigma(\pi)$ the number of occurrences of $\sigma$ in $\pi$ as a consecutive pattern. If $c_\sigma(\pi)=0$, we say that $\pi$ {\em avoids} $\sigma$. Let $\alpha_n(\sigma)$
be the number of permutations in $\S_n$ that avoid $\sigma$ as a consecutive pattern. The notions of occurrence, containment and avoidance in this paper always refer to consecutive patterns, even if it is not explicitly stated.

Let
$$P_{\sigma}(u,z)=\sum_{n\ge0} \sum_{\pi\in\S_n} u^{c_\sigma(\pi)}\frac{z^n}{n!}$$ be the exponential generating function for occurrences of $\sigma$ in permutations, and let
$\omega_\sigma(u,z)=1/P_{\sigma}(u,z)$.
Note that the generating function for permutations avoiding $\sigma$ is then
$$P_\sigma(0,z)=\frac{1}{\omega_\sigma(0,z)}=\sum_{n\ge0} \alpha_n(\sigma)\frac{z^n}{n!}.$$
When there is no confusion, we will write $\omega_\sigma(z)$ as a shorthand for $\omega_\sigma(0,z)$.
In the rest of the paper we assume that the length of the pattern $\sigma$ is $m\ge2$.

We denote by $\O_\sigma$ the set of overlaps of $\sigma$, which is defined as the set of indices $i$ with $1\le i<m$
such that $\red(\sigma_{i+1}\sigma_{i+2}\dots\sigma_m)=\red(\sigma_1\sigma_2\dots \sigma_{m-i})$.
Equivalently, $i\in\O_\sigma$ if two occurrences of $\sigma$ in a permutation can have starting positions at distance $i$ from each other.
Note that $m-1\in \O_\sigma$ for every $\sigma\in\S_m$. If $m\ge3$, a pattern $\sigma\in\S_m$ for which $\O_\sigma=\{m-1\}$ is said to be
{\em non-overlapping}. Equivalently, $\sigma$ is non-overlapping if two occurrences of $\sigma$ in a permutation cannot overlap in more than one position.
For example, the patterns $132$, $1243$, $1342$, $21534$ and $34671285$ are non-overlapping.
Non-overlapping patterns have been studied by Duane and Remmel~\cite{DR} and by B\'ona~\cite{Bon},
who shows that the proportion of non-overlapping patterns of any length $m$ is at least $0.364$.
It is easy to see that $1\in\O_\sigma$ if and only if $\sigma$ is monotone.

An important problem in permutation patterns is to determine when two patterns are avoided by the same number of permutations of length $n$ for every $n$ or, more generally, when
the same distribution of occurrences of the two patterns on permutations is the same. We discuss here only the case of consecutive patterns.
We say that two patterns $\sigma$ and $\tau$ are {\em strongly c-Wilf-equivalent} if $P_\sigma(u,z)=P_\tau(u,z)$, and that they are {\em c-Wilf-equivalent} if $P_\sigma(0,z)=P_\tau(0,z)$.
The last condition can be rephrased as $\alpha_n(\sigma)=\alpha_n(\tau)$ for all $n$. Nakamura~\cite[Conjecture~6]{Nak} conjectures that two patterns are strongly c-Wilf-equivalent iff they are c-Wilf-equivalent.
A complete classification into c-Wilf-equivalence classes is known for patterns of length up to 6, and in these cases they coincide with strong c-Wilf-equivalence classes.
It was shown in~\cite{EliNoy} that there are two equivalence classes of patterns of length~$3$, represented by the patterns $123$ and $132$, and seven classes of patterns of length~$4$, represented by
$1234$, $2413$, $2143$, $1324$, $1423$, $1342$, and $1243$. It was later proved in~\cite{Nak,EliNoy2} that there are $25$ classes for patterns of length~$5$, and $92$ for patterns of length~$6$.

It is clear that any pattern $\sigma=\sigma_1 \cdots \sigma_m$ is strongly c-Wilf-equivalent to its reversal $\sigma_m \cdots \sigma_1$ and its complementation $(m+1-\sigma_1) \cdots (m+1 - \sigma_m)$.
Using these operations, every $\sigma\in\S_m$ is strongly c-Wilf-equivalent to a pattern with $\sigma_1<\sigma_m$ and $\sigma_1+\sigma_m\le m+1$.

The main results of this paper, which settle three conjectures from~\cite{EliNoy}, \cite{Nak}, and~\cite{EliNoy2}, can be summarized as follows.

\bit\item For every $\sigma\in\S_m$ there exists $n_0$ such that $$\alpha_n(12\dots(m-2)m(m-1)) \le \alpha_n(\sigma)\le \alpha_n(12\dots m)$$ for all $n\ge n_0$.
\item For every non-overlapping $\sigma\in\S_m$, there exists $n_0$ such that $$\alpha_n(12\dots(m-2)m(m-1))\le\alpha_n(\sigma)\le\alpha_n(134\dots m2)$$ for all $n\ge n_0$.
\eit
These statements will be split into Theorem~\ref{thm:main}, Theorem~\ref{thm:nol}, and Theorem~\ref{thm:nak}, which will be proved in different sections.

\subsection{Asymptotic behavior}\label{sec:asym}

The results in this paper concern the asymptotic behavior of the sequences $\alpha_n(\sigma)$ for different patterns $\sigma$. When comparing their growth rates, the following result from~\cite{Eliasym} will be useful.

\begin{proposition}[\cite{Eliasym}]\label{prop:lim} For every $\sigma\in\S_m$ with $m\ge3$,
the limit $$\lim_{n\rightarrow\infty}\left(\frac{\alpha_n(\sigma)}{n!}\right)^{1/n}$$ exists, and it is strictly between $0$ and $1$.
\end{proposition}

We denote this limit by $\rho_\sigma$, and we call it the {\em growth rate} of $\sigma$. An elementary fact from singularity analysis of generating functions, called the
Exponential Growth Formula in~\cite[Theorem IV.7]{FS}, states in our case that $\rho_\sigma^{-1}$ is
the modulus of a singularity nearest to the origin (i.e. the radius of convergence) of $P_\sigma(0,z)$. Additionally, since
$P_\sigma(0,z)$ has non-negative coefficients, Pringsheim's Theorem~\cite[Theorem IV.6]{FS}
implies that this function has a real singularity at $z=\rho_\sigma^{-1}$.

\begin{theorem}[\cite{FS}]\label{thm:sing}
For every $\sigma\in\S_m$ with $m\ge3$,
$P_\sigma(0,z)$ has a singularity at $z=\rho_\sigma^{-1}$ and no singularities in $|z|<\rho_\sigma^{-1}$.
\end{theorem}

It is also shown in~\cite{Eliasym} that if $m\ge3$, then $\rho_\sigma\ge\min\{\rho_{123},\rho_{132}\}=\rho_{132}$. In the rest of the paper, we let $C=\rho_{132}^{-1}\approx 1.276$.
\begin{proposition}[\cite{Eliasym}]\label{prop:boundC} For every $\sigma\in\S_m$ with $m\ge3$, $$1<\rho_\sigma^{-1}\le C.$$\end{proposition}

Although we will not use it here, we remark that Ehrenborg, Kitaev and Perry~\cite{EKP} have given the following more accurate description of the asymptotic behavior of the sequences $\alpha_n(\sigma)$.
The proof of this important result relies on methods from spectral theory.

\begin{theorem}[\cite{EKP}]\label{thm:EKP}
For every $\sigma$,  $\alpha_n(\sigma)/n!=\gamma_\sigma\rho_\sigma^n+O(r_\sigma^n)$ for some constants $\gamma_\sigma$ and $r_\sigma<\rho_\sigma$.
\end{theorem}

\subsection{The cluster method}\label{sec:clustermethod}

The computation of the generating functions $P_{\sigma}(u,z)$ is simplified by using an adaptation of the cluster method of Goulden and Jackson~\cite{GJ79,GJ}, which is based on inclusion-exclusion.
We now summarize this adaptation to consecutive patterns in permutations, which has been recently used in~\cite{Dot,EliNoy2,KS}.

For fixed $\sigma\in\S_m$, a $k$-cluster of length $n$ with respect to $\sigma$ is a pair $(\pi;i_1,i_2,\dots,i_k)$ where the indices $i_j$ satisfy $1=i_1<i_2<\dots<i_k=n-m+1$ and $i_{j+1}\le i_j+m-1$ for all $j$,
and $\pi\in\S_n$ satisfies $\red(\pi_{i_j}\pi_{i_j+1}\dots\pi_{i_j+m-1})=\sigma$ for all $j$.
In other words, the $i_j$ are starting positions of occurrences of $\sigma$ in $\pi$, all the entries of $\pi$ belong to at least one of these marked occurrences, and neighboring marked occurrences overlap.
For example, if $\sigma=1324$, then $(142536879;1,3,6)$ is a $3$-cluster of length $9$.
Note that $i_{j+1}-i_j\in\O_\sigma$ (the set of overlaps) for all $j$, and that $\pi$ may have additional occurrences of $\sigma$ aside from the marked ones.

Let $r^\sigma_{n,k}$ denote the number of $k$-clusters of length $n$ with respect to $\sigma$. For example, $r^\sigma_{m,1}=1$ for any $\sigma\in\S_m$, and $r^{132}_{5,2}=3$ because of the clusters
$(13254;1,3)$, $(14253;1,3)$ and $(15243;1,3)$. More examples of $r^\sigma_{n,k}$ are given in Table~\ref{tab:r23}. Let
$$R_\sigma(t,z)=\sum_{n,k}r^\sigma_{n,k}t^k\frac{z^n}{n!}$$ be the exponential generating function for clusters. The cluster method \cite[Theorem 2.8.6]{GJ}, adapted to permutations,
can be stated as follows.

\begin{theorem}[\cite{GJ}]\label{thm:GJ} For every $\sigma\in\S_m$,
$$\omega_\sigma(u,z)=1-z-R_\sigma(u-1,z).$$
\end{theorem}

Because of the above theorem, finding the generating function $P_{\sigma}(u,z)$ for occurrences of $\sigma$ in permutations is equivalent to computing the cluster numbers $r^\sigma_{n,k}$.
The advantage of these numbers is that they can be interpreted as counting linear extensions of certain posets, as shown in~\cite{EliNoy2}.
Given $\sigma\in\S_m$ and $k$, let
$$\I^\sigma_{k}=\{(i_1,i_2,\dots,i_k) : i_1=1 \mbox{ and } i_{j+1}-i_j\in\O_\sigma \mbox{ for }1\le j\le k-1\}$$
be the set of possible tuples of starting positions of marked occurrences of $\sigma$ in $k$-clusters.
If $(i_1,\dots,i_k)\in\I^\sigma_{k}$, then
$(\pi;i_1,\dots,i_k)$ is a $k$-cluster with respect to $\sigma$ if and only if
$\pi\in\S_{i_k+m-1}$ and, for each $1\le j\le k$,
\beq\label{eq:occurrence}\red(\pi_{i_j}\pi_{i_j+1}\dots\pi_{i_j+m-1})=\sigma.\eeq
Denoting by $\varsigma\in\S_m$ the inverse of $\sigma$, condition~\eqref{eq:occurrence} is equivalent to
\beq\label{eq:occurrence2}\pi_{\varsigma_1+i_j-1}<\pi_{\varsigma_2+i_j-1}<\dots<\pi_{\varsigma_m+i_j-1}.\eeq
The inequalities~\eqref{eq:occurrence2} for $1\le j\le k$ define a partial order on the set $\{\pi_1,\pi_2,\dots,\pi_{i_k+m-1}\}$.
This partially ordered set (poset) is denoted by $Q^\sigma_{i_1,\dots,i_k}$ and called a {\em cluster poset}.
Denoting by $\lext(Q)$ the set of linear extensions (i.e., compatible linear orders) of a poset $Q$, it is clear that
$(\pi;i_1,\dots,i_k)$ is a $k$-cluster with respect to $\sigma$ if and only if $\pi\in\lext(Q^\sigma_{i_1,\dots,i_k})$.
To specify the length $n$ of the cluster, we let $\I^\sigma_{n,k}=\{(i_1,\dots,i_k)\in\I^\sigma_{k}:i_k=n-m+1\}$, so that
\beq\label{eq:rnk}r^\sigma_{n,k}=\sum_{(i_1,\dots,i_k)\in\I^\sigma_{n,k}}\lext(Q^\sigma_{i_1,\dots,i_k}).\eeq

For example, if $\sigma=14253$, then $\O_\sigma=\{2,4\}$, so $(1,3,7)\in\I^\sigma_3$. In this case, $(\pi;1,3,7)$ is a $3$-cluster
if $\pi\in\S_{11}$ and the following inequalities hold:
$$\pi_1<\pi_3<\pi_5<\pi_2<\pi_4, \quad \pi_3<\pi_5<\pi_7<\pi_4<\pi_6, \quad \pi_7<\pi_9<\pi_{11}<\pi_8<\pi_{10}.$$
Equivalently, $\pi$ is a linear extension of the poset $Q^{\sigma}_{1,3,7}$ drawn on the left of Figure~\ref{fig:lext}. An example of a linear extension is given on the right,
corresponding to $\pi=1\ 6\ 2\ 8\ 3\ 11 \ 4\ 9\ 5\ 10\ 7$.

\begin{figure}[htb]
\centering
\begin{tikzpicture}[scale=.6]
\fill (1,0) circle (0.1) node[right]{$\pi_1$};
\fill (1,1) circle (0.1) node[right]{$\pi_3$};
\fill (1,2) circle (0.1) node[right]{$\pi_5$};
\fill (0,3) circle (0.1) node[left]{$\pi_2$};
\fill (2,3) circle (0.1) node[right]{$\pi_7$};
\fill (1,4) circle (0.1) node[right]{$\pi_4$};
\fill (1,5) circle (0.1) node[right]{$\pi_6$};
\fill (3,4) circle (0.1) node[right]{$\pi_9$};
\fill (3,5) circle (0.1) node[right]{$\pi_{11}$};
\fill (3,6) circle (0.1) node[right]{$\pi_8$};
\fill (3,7) circle (0.1) node[right]{$\pi_{10}$};
\draw (1,0)--(1,2)--(0,3)--(1,4)--(1,5);
\draw (1,2)--(3,4)--(3,7);
\draw (2,3)--(1,4);
\end{tikzpicture}\hspace{15mm}
\begin{tikzpicture}[scale=.6]
\fill (1,0) circle (0.1) node[right]{$1$};
\fill (1,1) circle (0.1) node[right]{$2$};
\fill (1,2) circle (0.1) node[right]{$3$};
\fill (0,3) circle (0.1) node[left]{$6$};
\fill (2,3) circle (0.1) node[right]{$4$};
\fill (1,4) circle (0.1) node[left]{$8$};
\fill (1,5) circle (0.1) node[left]{$11$};
\fill (3,4) circle (0.1) node[right]{$5$};
\fill (3,5) circle (0.1) node[right]{$7$};
\fill (3,6) circle (0.1) node[right]{$9$};
\fill (3,7) circle (0.1) node[right]{$10$};
\draw (1,0)--(1,2)--(0,3)--(1,4)--(1,5);
\draw (1,2)--(3,4)--(3,7);
\draw (2,3)--(1,4);
\end{tikzpicture}
\caption{\label{fig:lext} The poset $Q^{14253}_{1,3,7}$ and a linear extension.}
\end{figure}
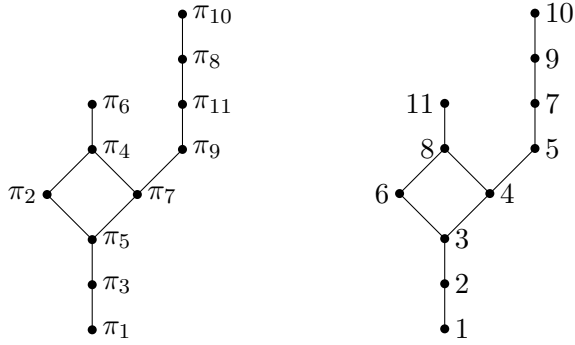

\section{The most avoided pattern}\label{sec:most}

In this section we prove the CMP conjecture, which is stated as Theorem~\ref{thm:main} below. The proof involves a detailed analysis of the functions $\omega_\sigma(z)$ which, by Theorem~\ref{thm:GJ},
are closely related to the exponential generating functions for clusters. For the monotone pattern, we have a simple alternating series expansion.
\begin{proposition}[\cite{GJ,EliNoy2}] We have
$$\omega_{12\dots m}(z)=\sum_{j\ge0}\frac{z^{jm}}{(jm)!}-\sum_{j\ge0}\frac{z^{jm+1}}{(jm+1)!}.$$
\end{proposition}

Since the terms of the above alternating series are decreasing in absolute value when $0<z\le C$, we get the following upper bound.

\begin{proposition}\label{prop:ineqmon}
For $0<z\le C$,
$$\omega_{12\dots m}(z)<1-z+\frac{z^m}{m!}-\frac{z^{m+1}}{(m+1)!}+\frac{z^{2m}}{(2m)!}.$$
\end{proposition}

For an arbitrary pattern, the generating function $\omega_\sigma(z)$ can also be expressed as an alternating sum, although the coefficients are not as simple as for the monotone pattern. The trick is to write
$$R_\sigma(t,z)=\sum_{n,k}r^\sigma_{n,k}t^k\frac{z^n}{n!}=\sum_{k\ge1}\left(\sum_n r^\sigma_{n,k}\frac{z^n}{n!}\right)t^k=\sum_{k\ge1}s^\sigma_k(z)t^k,$$
where we define $$s^\sigma_k(z)=\sum_n r^\sigma_{n,k}\frac{z^n}{n!}.$$
In particular,
\beq\label{eq:s12} s^\sigma_1(z)=\frac{z^m}{m!} \qquad \mbox{and} \qquad s^\sigma_2(z)=\sum_{\ell\in\O_\sigma} r^\sigma_{m+\ell,2}\frac{z^{m+\ell}}{(m+\ell)!},\eeq
since all $2$-clusters are of the form $(\pi;1,\ell+1)$ with $\ell\in\O_\sigma$.
 By Theorem~\ref{thm:GJ}, \beq\label{eq:omegaalt}\omega_\sigma(z)=1-z-R_\sigma(-1,z)=1-z-\sum_{k\ge1}s^\sigma_k(z)(-1)^k,\eeq
which has the advantage of being an alternating sum. To obtain bounds for $\omega_\sigma(z)$ similar to Proposition~\ref{prop:ineqmon},
we will show that the terms of this sum decrease in absolute value. First we state an easy lemma that will be used in the proof.

\begin{lemma}\label{lem:m4}
If $\sigma\in\S_m\setminus\{12\dots m,m\dots 21\}$ and $2,3\in\O_\sigma$, then $m=4$.
\end{lemma}

\begin{proof}
The fact that $3\in\O_\sigma$ implies that $m\ge4$. Suppose that $m\ge 5$.
 Without loss of generality, we can assume that $\sigma_1<\sigma_2$.
If $\sigma_2<\sigma_3$, then the fact that $2\in\O_\sigma$ would imply that $\sigma=12\dots m$, so we must have $\sigma_2>\sigma_3$.
Since $2\in\O_\sigma$, it follows that $\sigma_4>\sigma_5$, but since $\sigma_1<\sigma_2$ and $3\in\O_\sigma$, we also must have $\sigma_4<\sigma_5$, which is a contradiction.
\end{proof}

\begin{proposition}\label{prop:decreasing}
For every $\sigma\in\S_m$ and $0<z\le C$, the sequence $$\{s^\sigma_k(z)\}_{k\ge1}$$ is decreasing.
\end{proposition}

\begin{proof}
From the definition of $s^\sigma_k(z)$ and equation~\eqref{eq:rnk}, we have
\beq\label{eq:sk}s^\sigma_k(z)=\sum_n r^\sigma_{n,k}\frac{z^n}{n!}=\sum_{(i_1,i_2,\dots,i_k)\in\I^\sigma_{k}}\lext(Q^\sigma_{i_1,i_2,\dots,i_k})\frac{z^{i_k+m-1}}{(i_k+m-1)!}.\eeq
To compare $s^\sigma_{k+1}(z)$ and $s^\sigma_k(z)$, we use a natural surjective map, which we denote by $\Gamma$, from $k{+}1$-clusters to $k$-clusters.
This map consists of deleting the part of the permutation to the right of the $k$-th marked occurrence,
namely $$\Gamma:(\pi_1\pi_2\dots\pi_{i_{k+1}+m-1};i_1,\dots,i_k,i_{k+1})\mapsto(\red(\pi_1\pi_2\dots\pi_{i_{k}+m-1});i_1,\dots,i_k).$$

Fix a $k$-cluster $(\pi_1\pi_2\dots\pi_n;i_1,\dots,i_k)$, where we let $n=i_{k}+m-1$, and fix $\ell\in\O_\sigma$.
The number of $k{+}1$-clusters of length $n+\ell$ that are mapped by $\Gamma$ to the fixed $k$-cluster is clearly bounded from above by $\binom{n+\ell}{\ell}$, since
such a $k{+}1$-cluster is uniquely determined by choosing the subset of $\{1,2,\dots,n+\ell\}$ corresponding to the values of the entries
$\{\pi_{n+1},\dots,\pi_{n+\ell}\}$. In fact, although not used in this proof, this bound can be improved to $\binom{n-m+2\ell}{\ell}$,
since the order of the entries $\{\pi_{n+1},\dots,\pi_{n+\ell}\}$ needs to be determined only in relation to the entries $\{\pi_{1},\pi_{2},\dots,\pi_{n-m+\ell}\}$.
In other words, each linear extension of $Q^\sigma_{i_1,\dots,i_k}$ can be {\it extended} in at most $\binom{n-m+2\ell}{\ell}$ ways to a linear extension of $Q^\sigma_{i_1,\dots,i_k,i_k+\ell}$, so
\beq\label{eq:n+l}\lext(Q^\sigma_{i_1,\dots,i_k,i_k+\ell})
\le\binom{n-m+2\ell}{\ell}\lext(Q^\sigma_{i_1,\dots,i_k})
\le\binom{n+\ell}{\ell}\lext(Q^\sigma_{i_1,\dots,i_k}).\eeq
Since the inequality
$$\lext(Q^\sigma_{i_1,\dots,i_k,i_k+\ell})\frac{z^{n+\ell}}{(n+\ell)!}\le\lext(Q^\sigma_{i_1,\dots,i_k})\frac{z^n}{n!} \frac{z^\ell}{\ell!}$$
holds for every $\ell\in\O_\sigma$, we get
$$\sum_{\ell\in\O_\sigma}\lext(Q^\sigma_{i_1,\dots,i_k,i_k+\ell})\frac{z^{n+\ell}}{(n+\ell)!}
\le\lext(Q^\sigma_{i_1,\dots,i_k})\frac{z^n}{n!} \sum_{\ell\in\O_\sigma} \frac{z^\ell}{\ell!}.$$
Summing both sides of the last inequality over all $(i_1,\dots,i_k)\in\I^\sigma_k$ and using~\eqref{eq:sk}, we get
\beq\label{eq:skcompare} s^\sigma_{k+1}(z)\le s^\sigma_k(z) \sum_{\ell\in\O_\sigma} \frac{z^\ell}{\ell!}.\eeq
It remains to bound the sum
$$\sum_{\ell\in\O_\sigma} \frac{z^\ell}{\ell!}\le \sum_{\ell\in\O_\sigma} \frac{C^\ell}{\ell!}.$$

Suppose first that $\sigma$ is not monotone, so $1\notin\O_\sigma$. If $\{2,3\}\nsubseteq \O_\sigma$,
then
\beq\label{eq:97a}\sum_{\ell\in\O_\sigma} \frac{C^\ell}{\ell!}\le \frac{C^2}{2!}+\sum_{\ell\ge4} \frac{C^\ell}{\ell!}=e^{C}-1-C-\frac{C^3}{6}<1,\eeq
and so $s^\sigma_{k+1}(z)\le s^\sigma_k(z)$ as desired.
If $\{2,3\}\subseteq \O_\sigma$, then $m=4$ by Lemma~\ref{lem:m4}, and $\sigma$ is c-Wilf equivalent to one of $2413$, $2143$, $1324$, or $1423$.
For each one of these patterns, the bound $\binom{n+\ell}{\ell}$ used in~\eqref{eq:n+l} can be improved individually. For the rest of the argument to carry over,
it is enough to give, for each $\sigma\in\{2413,2143,1324,1423\}$, upper bounds $h_\ell^\sigma$ for $\ell=2,3$ on number of ways to extend a $k$-cluster of length $n$ to a $k{+}1$-cluster of length $n+\ell$, satisfying
\beq\label{eq:97b}\frac{h_2^\sigma\, C^2}{(n+2)(n+1)}+\frac{h_3^\sigma\, C^3}{(n+3)(n+2)(n+1)}<1\eeq
for $n\ge 4$. For each individual pattern, at least one of the bounds $h^\sigma_2=\binom{n+2}{2}$ and $h^\sigma_3=\binom{n+3}{3}$ used in~\eqref{eq:n+l} can be improved as follows to satisfy inequality~\eqref{eq:97b}:
for $\sigma=2413$ we have $h^\sigma_2=\frac{(n+1)^2}{4}$,
for $\sigma=2143$ we have $h^\sigma_2=1$, for $\sigma=1324$ we have $h^\sigma_3=1$, and for $\sigma=1423$ we have $h^\sigma_2=1$. The details are left to the reader.

If $\sigma$ is monotone, then for every $(i_1,i_2,\dots,i_k)\in\I^\sigma_{k}$, the poset $Q^\sigma_{i_1,\dots,i_k}$ is a chain, so $\lext(Q^\sigma_{i_1,\dots,i_k})=1$.
In particular, for $n\ge m$,
$$\lext(Q^\sigma_{i_1,\dots,i_k,i_k+\ell})\frac{z^{n+\ell}}{(n+\ell)!}=\lext(Q^\sigma_{i_1,\dots,i_k})\frac{z^n}{n!} \frac{z^\ell}{(n+\ell)_\ell}
\le\lext(Q^\sigma_{i_1,\dots,i_k})\frac{z^n}{n!} \frac{z^\ell}{(m+\ell)_\ell},$$
where we use the notation $(a)_\ell=a(a-1)\dots (a-\ell+1)$. Summing over $\ell\in\O_\sigma=\{1,2\dots,m-1\}$ and over all $(i_1,\dots,i_k)\in\I^\sigma_k$, we get
$$s^\sigma_{k+1}(z)\le s^\sigma_k(z) \sum_{\ell=1}^{m-1} \frac{z^\ell}{(m+\ell)_\ell}.$$
The fact that $\{s^\sigma_k(z)\}_{k\ge1}$ is decreasing follows now from the inequalities
\beq\label{eq:97c}\sum_{\ell=1}^{m-1} \frac{z^\ell}{(m+\ell)_\ell}\le \sum_{\ell=1}^{m-1} \frac{C^\ell}{(m+\ell)_\ell}\le \sum_{\ell=1}^{m-1} \frac{C^\ell}{(m+1)^\ell}< \frac{\frac{C}{m+1}}{1-\frac{C}{m+1}}<1.\eeq
\end{proof}

\begin{remark}
The argument in the proof of Proposition~\ref{prop:decreasing} shows also that $s^\sigma_{k+1}(z)/s^\sigma_{k}(z)<0.97$ for all $k\ge1$, since inequalities \eqref{eq:97a},
\eqref{eq:97b} and \eqref{eq:97c} also hold when substituting $0.97$ for $1$.
\end{remark}

We can now give bounds on $\omega_\sigma(z)$ for an arbitrary pattern.

\begin{proposition}\label{prop:boundsomega}
For every $\sigma\in\S_m$ and $0<z\le C$, \begin{align*} 1-z+\frac{z^m}{m!}-s^\sigma_2(z)<\omega_\sigma(z)&<1-z+\frac{z^m}{m!}-s^\sigma_2(z)+s^\sigma_3(z)\\ &<1-z+\frac{z^m}{m!}.\end{align*}
\end{proposition}

\begin{proof}
By Proposition~\ref{prop:decreasing},
$$-s^\sigma_1(z)<-s^\sigma_1(z)+s^\sigma_2(z)-s^\sigma_3(z)<\sum_{k\ge1}(-1)^k s^\sigma_k(z)<-s^\sigma_1(z)+s^\sigma_2(z)$$
for $0<z\le C$, since the terms of the above alternating series decrease in absolute value.
Now we use equations~\eqref{eq:s12} and~\eqref{eq:omegaalt}.
\end{proof}

\begin{corollary}\label{cor:nosing}
For every $\sigma\in\S_m$, $\omega_\sigma(z)$ is analytic in $|z|\le C$.
\end{corollary}

\begin{proof}
By the remark following Proposition~\ref{prop:decreasing}, $s^\sigma_k(C)\le s^\sigma_1(C)0.97^{k-1}< 0.97^k$ for all $k\ge1$.
Thus, for $|z|\le C$,
$$|s^\sigma_k(z)(-1)^k|\le \sum_n r^\sigma_{n,k}\frac{|z|^n}{n!}\le \sum_n r^\sigma_{n,k}\frac{C^n}{n!}=s^\sigma_k(C)<0.97^k,$$
and so the series~\eqref{eq:omegaalt} converges.
\end{proof}

\begin{corollary}\label{cor:smallestzero}
For every $\sigma\in\S_m$ with $m\ge3$, $\omega_\sigma(z)$ has a zero at $z=\rho_\sigma^{-1}$ and no zeroes in $|z|<\rho_\sigma^{-1}$. 
\end{corollary}
\begin{proof}
By Theorem~\ref{thm:sing}, $z=\rho_\sigma^{-1}$ is a singularity of $P_\sigma(0,z)$ nearest to the origin, and it satisfies $\rho_\sigma^{-1}\le C$ by Proposition~\ref{prop:boundC}.
By Corollary~\ref{cor:nosing}, the only singularities of $P_\sigma(0,z)=1/\omega_\sigma(z)$ in $|z|\le C$ are zeroes of $\omega_\sigma(z)$,
so in particular $\rho_\sigma^{-1}$ is a zero nearest to the origin.
\end{proof}

Using Proposition~\ref{prop:boundsomega}, the bound from Proposition~\ref{prop:boundC} can be improved for patterns of length at least $4$. In the rest of the paper, we denote by $c$ the smallest positive zero of $1-z+z^4/24$. Note that $c\approx 1.051$.

\begin{corollary}\label{cor:c1}
For every $\sigma\in\S_m$ with $m\ge4$, $$1<\rho_\sigma^{-1}< c.$$
\end{corollary}
\begin{proof}
Let $\tau=\red(\sigma_1\sigma_2\sigma_3\sigma_4)\in\S_4$. Clearly, every permutation avoiding $\tau$ must avoid $\sigma$ as well, so
$\alpha_n(\tau)\le\alpha_n(\sigma)$ for all $n$. It follows that $\rho_\tau\le \rho_\sigma$, so $\rho_\sigma^{-1}$ is bounded from above by
$\rho_\tau^{-1}$, which by Corollary~\ref{cor:smallestzero} is the smallest positive zero of $\omega_\tau(z)$, and by Proposition~\ref{prop:boundC} satisfies $\rho_\tau^{-1}\le C$.
 By Proposition~\ref{prop:boundsomega}, $$\omega_\tau(z)<1-z+\frac{z^4}{24}$$ for $0<z\le C$. Since $\omega_\tau(0)=1$, the smallest positive zero of
$\omega_\tau(z)$ must be to the left of $c$. The fact that $1<\rho_\sigma^{-1}$ follows from Proposition~\ref{prop:boundC}.
\end{proof}

The last ingredient that we need to prove our main theorem is a bound on the number of $2$-clusters.

\begin{lemma}\label{lem:2l}
For every $\sigma\in\S_m$ and $\ell\in\O_\sigma$, $$r^\sigma_{m+\ell,2}\le\binom{2\ell-1}{\ell-1}.$$
\end{lemma}

\begin{proof}
Recall that $r^\sigma_{m+\ell,2}$ is the number of linear extensions of $Q^\sigma_{1,\ell+1}$. Letting $\varsigma=\sigma^{-1}$, this poset consists of two chains of length $m$,
\beq\label{eq:2chains} \pi_{\varsigma_1}<\pi_{\varsigma_2}<\dots<\pi_{\varsigma_m} \quad \mbox{and} \quad \pi_{\varsigma_1+\ell}<\pi_{\varsigma_2+\ell}<\dots<\pi_{\varsigma_m+\ell},\eeq
sharing $m-\ell$ elements $\pi_{\ell+1},\dots,\pi_{m}$. 
Denote by $A=\{\pi_1,\pi_2,\dots,\pi_\ell\}$ and $B=\{\pi_{m+1},\pi_{m+2},\dots,\pi_{m+\ell}\}$ the sets of elements in each chain that are not in the other chain. 

Let $a_0$ (resp. $b_0$) be the number of elements of $A$ (resp. $B$) that are less than $\pi_m$ in $Q^\sigma_{1,\ell+1}$.
Then
$$r^\sigma_{m+\ell,2}\le \binom{a_0+b_0}{a_0}\binom{2\ell-a_0-b_0}{\ell-a_0},$$
because a linear extension of $Q^\sigma_{1,\ell+1}$ is determined by the order of the elements of $A$
relative to the elements of $B$, but only elements below (resp. above) $\pi_m$ need to be compared with each other.

By symmetry, we can assume that $a_0+b_0\le \ell$ and $a_0\le b_0$. Besides, we claim that $a_0\neq b_0$,
because otherwise $\pi_m$ would be in the same relative position in each of the two chains~\eqref{eq:2chains}, implying that $m=m+\ell$, which is a contradiction.
Thus, an upper bound on $r^\sigma_{m+\ell,2}$ is given by
$$\underset{a_0+b_0\le \ell}{\max_{0\le a_0<b_0\le \ell}} \binom{a_0+b_0}{a_0}\binom{2\ell-a_0-b_0}{\ell-a_0}.$$
Setting $p=\ell+1$, $a=a_0+1$, $b=b_0+1$, this expression becomes
$$\underset{a+b\le p+1}{\max_{1\le a<b\le p}} \binom{a+b-2}{a-1}\binom{2p-a-b}{p-a}=\binom{2p-3}{p-2}=\binom{2\ell-1}{\ell-1},$$
using equation~\eqref{eq:maxf}, which will be proved later.
\end{proof}

We are now ready to prove the CMP Conjecture.

\begin{theorem}\label{thm:main}
For every $\sigma\in\S_m\setminus\{12\dots m,m\dots 21\}$, there exists $n_0$ such that $$\alpha_n(\sigma)<\alpha_n(12\dots m)$$ for all $n\ge n_0$.
\end{theorem}

\begin{proof}
The case $m=3$ was proved in~\cite{EliNoy}, so we will assume for simplicity that $m\ge4$.
Let $\sigma\in\S_m\setminus\{12\dots m,m\dots 21\}$.
We will prove that $\rho_\sigma<\rho_{12\dots m}$, which is equivalent to the statement of the theorem.
By Corollary~\ref{cor:smallestzero}, $\rho_\sigma^{-1}$ and $\rho_{12\dots m}^{-1}$ are the smallest positive zeroes of $\omega_\sigma(z)$ and $\omega_{12\dots m}(z)$, respectively.
By Corollary~\ref{cor:c1}, their values lie between $1$ and $c$. Since $\omega_\sigma(0)=\omega_{12\dots m}(0)=1$,
the inequality $\rho_{12\dots m}^{-1}<\rho_\sigma^{-1}$ will be a consequence of the fact that \beq\label{eq:omegamonsigma}\omega_{12\dots m}(z)<\omega_\sigma(z)\eeq for $0<z<c$.

By Proposition~\ref{prop:ineqmon}, the lower bound in Proposition~\ref{prop:boundsomega}, and equation~\eqref{eq:s12}, inequality~\eqref{eq:omegamonsigma} will follow if we show that
$$\sum_{\ell\in\O_\sigma} r^\sigma_{m+\ell,2}\frac{z^{m+\ell}}{(m+\ell)!} <\frac{z^{m+1}}{(m+1)!}-\frac{z^{2m}}{(2m)!}$$
for $0<z<c$, which in turn follows from
\beq\label{eq:ineqr}\sum_{\ell\in\O_\sigma} r^\sigma_{m+\ell,2}\frac{c^{\ell-1}}{(m+\ell)_{\ell-1}}+\frac{c^{m-1}}{(2m)_{m-1}}<1.\eeq

Using Lemma~\ref{lem:2l} and the fact that $1\notin\O_\sigma$, 
we get
\beq\label{eq:sumr2}\sum_{\ell\in\O_\sigma} r^\sigma_{m+\ell,2}\frac{c^{\ell-1}}{(m+\ell)_{\ell-1}}
\le \sum_{\ell\in\O_\sigma} \binom{2\ell-1}{\ell-1}\frac{c^{\ell-1}}{(m+\ell)_{\ell-1}}
\le \sum_{\ell=2}^{m-1} \binom{2\ell-1}{\ell-1}\frac{c^{\ell-1}}{(m+\ell)_{\ell-1}}.\eeq
The fact that the last term of the sum on the right, corresponding to $\ell=m-1$, equals
$$\frac{m(m+1)c^{m-2}}{2(2m-1)(m-1)!},$$
suggests that we define
$$g(m)=\sum_{\ell=2}^{m-1} \binom{2\ell-1}{\ell-1}\frac{c^{\ell-1}}{(m+\ell)_{\ell-1}}+\sum_{j\ge m}\frac{(j+1)(j+2)c^{j-1}}{2(2j+1)j!},$$
which clearly bounds~\eqref{eq:sumr2} from above.
Comparing the above expressions for $g(m)$ and $g(m+1)$ term by term, it is clear that $g(m)>g(m+1)$ for all $m$. Indeed,
when going from $g(m)$ to $g(m+1)$, the first $m-2$ terms of the sum become smaller and the others stay equal.
It follows that for $m\ge4$,  $g(m)\le g(4)<0.9$,
where the last inequality is obtained by bounding the infinite sum in $g(4)$ by
$$\sum_{j\ge 4}\frac{(j+1)(j+2)c^{j-1}}{2(2j+1)j!}<\frac{1}{4}\sum_{j\ge 4}\frac{(j+3)c^{j-1}}{j!}=\frac{1}{4}\left(e^{c}(1+\frac{3}{c})-\frac{3}{c}-4-\frac{5c}{2}-c^2\right).$$
Using now that $\frac{c^{m-1}}{(2m)_{m-1}}<0.1$ for $m\ge4$, inequality~\eqref{eq:ineqr} is proved.
\end{proof}

\section{Non-overlapping patterns}\label{sec:nol}

Recall that $\sigma\in\S_m$ is non-overlapping if $\O_\sigma=\{m-1\}$ and $m\ge3$.
We denote by $\N_m$ the set of non-overlapping patterns in $\S_m$. These patterns have been considered recently by
Duane and Remmel~\cite{DR} and by B\'ona~\cite{Bon}, who shows that $|\N_m|/|\S_m|>0.364$ for all $m$.

In this section we study non-overlapping patterns with two purposes. On one hand, we prove Conjecture~7.1 from~\cite{EliNoy2},
stating that among non-overlapping patterns of length $m$, the pattern $134\dots m2$ is the most avoided one, while $12\dots(m-2)m(m-1)$ is the least avoided one.
This is stated as Theorem~\ref{thm:nol} below.
On the other hand, the fact that $12\dots(m-2)m(m-1)$ is the least avoided non-overlapping pattern of length $m$ will be a significant part of our proof, in Section~\ref{sec:least}, that
this pattern is also the least avoided among {\em all} patterns of length $m$. This was conjectured by Nakamura~\cite{Nak}, and will be proved in Theorem~\ref{thm:nak}.

For $\sigma\in\N_m$, $k$-clusters with respect to $\sigma$ must have length $k(m-1)+1$. In fact, $(\pi;i_1,\dots,i_k)$ is a $k$-cluster if and only if
$i_j=(j-1)(m-1)+1$ for $1\le j\le k$, and $\pi$ is a linear extension of the poset $Q^\sigma_{1,m,2m-1,\dots,(k-1)(m-1)+1}$, which we denote by $D^\sigma_k$ for simplicity.
Letting $a=\sigma_1$, $b=\sigma_m$ and $\varsigma=\sigma^{-1}$, the poset $D^\sigma_k$ consists of $k$ chains $$\pi_{\varsigma_1+(j-1)(m-1)}<\pi_{\varsigma_2+(j-1)(m-1)}<\dots<\pi_{\varsigma_m+(j-1)(m-1)}$$ for $1\le j\le k$,
where the $j$-th and $j{+}1$-st chains share one element $\pi_{\varsigma_b+(j-1)(m-1)}=\pi_{i_{j+1}}=\pi_{\varsigma_a+j(m-1)}$.
An example is drawn in Figure~\ref{fig:posetnol}. We denote the number of linear extensions of $D^\sigma_k$, which is the number of $k$-clusters with respect to $\sigma$, by
$d^\sigma_k=r^\sigma_{k(m-1)+1,k}$. Thus, in the non-overlapping case,
\beq\label{eq:snol} s^\sigma_k(z)=d^\sigma_k\frac{z^{k(m-1)+1}}{(k(m-1)+1)!},\eeq
and by Theorem~\ref{thm:GJ},
\beq\label{eq:omeganol}\omega_\sigma(u,z)=1-z-\sum_{k\ge1}(u-1)^k d^\sigma_k\frac{z^{k(m-1)+1}}{(k(m-1)+1)!}.\eeq
It is clear from this construction that the poset $D^\sigma_k$ and the numbers $d^\sigma_k$ depend only on $m$, $\sigma_1$ and $\sigma_m$, but not on the rest of the entries of~$\sigma$ (as long as it is non-overlapping), and consequently
so does $P_\sigma(u,z)=1/\omega_\sigma(u,z)$. This had been conjectured in~\cite{Elitesis} and has been proved by Dotsenko and Khoroshkin~\cite{Dot}, and independently by Duane and Remmel~\cite{DR}.

\begin{figure}[htb]
\centering
\begin{tikzpicture}[scale=.5]
\fill (0,0) circle (0.1) node[left]{$\pi_{\varsigma_1}$};
\fill (0,1) circle (0.1) node[left]{$\pi_{\varsigma_2}$};
\fill (0,2) circle (0.1) ;
\fill (0,3) circle (0.1) ;
\fill (0,4) circle (0.1) node[left]{$\pi_{\varsigma_{b-1}}$};
\fill (0,5) circle (0.1) node[left]{$\pi_{\varsigma_b}=\pi_m=\pi_{\varsigma_{a}+m-1}$};
\fill (0,6) circle (0.1) node[left]{$\pi_{\varsigma_{b+1}}$};
\fill (0,7) circle (0.1) ;
\fill (0,8) circle (0.1) ;
\fill (0,9) circle (0.1) node[left]{$\pi_{\varsigma_m}$};

\fill (1,2) circle (0.1) node[right]{$\pi_{\varsigma_1+m-1}$};
\fill (1,3) circle (0.1) ;
\fill (1,4) circle (0.1) ;
\fill (1,6) circle (0.1) ;
\fill (1,7) circle (0.1) ;
\fill (1,8) circle (0.1) ;
\fill (1,9) circle (0.1) ;
\fill (1,10) circle (0.1) ;
\fill (1,11) circle (0.1) node[left]{$\pi_{\varsigma_m+m-1}$};

\fill (2,4) circle (0.1) node[right]{$\pi_{\varsigma_1+2(m-1)}$};
\fill (2,5) circle (0.1) ;
\fill (2,6) circle (0.1) ;
\fill (2,8) circle (0.1) ;
\fill (2,9) circle (0.1) ;
\fill (2,10) circle (0.1) ;
\fill (2,11) circle (0.1) ;
\fill (2,12) circle (0.1) ;
\fill (2,13) circle (0.1) node[left]{$\pi_{\varsigma_m+2(m-1)}$};

\fill (3,6) circle (0.1) node[right]{$\pi_{\varsigma_1+(k-1)(m-1)}$};
\fill (3,7) circle (0.1) ;
\fill (3,8) circle (0.1) ;
\fill (3,10) circle (0.1) ;
\fill (3,11) circle (0.1) ;
\fill (3,12) circle (0.1) ;
\fill (3,13) circle (0.1) ;
\fill (3,14) circle (0.1) ;
\fill (3,15) circle (0.1) node[right]{$\pi_{\varsigma_m+(k-1)(m-1)}$};

\draw (0,0)--(0,9);
\draw[color=red] (1,2)--(1,4)--(0,5)--(1,6)--(1,11);
\draw[color=blue] (2,4)--(2,6)--(1,7)--(2,8)--(2,13);
\draw[color=brown] (3,6)--(3,8)--(2,9)--(3,10)--(3,15);
\end{tikzpicture}
\caption{\label{fig:posetnol} The poset $D^\sigma_k$ for $\sigma\in\N_m$ with $\sigma_1=a$ and $\sigma_m=b$. In this picture, $k=4$, $m=10$, $a=4$ and $b=6$.}
\end{figure}

\begin{lemma}[\cite{Dot,DR}]\label{lem:1m}
If $\sigma,\tau\in\N_m$ are such that $\sigma_1=\tau_1$ and $\sigma_m=\tau_m$, then $\sigma$ and $\tau$ are strongly c-Wilf equivalent.
\end{lemma}

Although we do not have a closed formula for $d^\sigma_k$ in general, it is clear that $d^\sigma_1=1$ and $$d^\sigma_2=\binom{\sigma_1+\sigma_m-2}{\sigma_1-1}\binom{2m-\sigma_1-\sigma_m}{m-\sigma_m}$$ for $\sigma\in\N_m$.
It will be convenient to define
$$f(a,b)=\binom{a+b-2}{a-1}\binom{2m-a-b}{m-b}$$
for $1\le a,b\le m$, so that $d^\sigma_2=f(\sigma_1,\sigma_m)$ for any $\sigma\in\N_m$.

We start by proving that another conjecture of Nakamura~\cite[Conjecture~6]{Nak} holds in the special case of non-overlapping patterns.

\begin{lemma}\label{lem:strong}
Two non-overlapping patterns are c-Wilf-equivalent iff they are strongly c-Wilf-equivalent.
\end{lemma}

\begin{proof}
Since $P_\sigma(u,z)=1/\omega_\sigma(u,z)$, two patterns $\sigma$ and $\tau$ are c-Wilf-equivalent iff $\omega_\sigma(0,z)=\omega_\tau(0,z)$,
and they are strongly c-Wilf-equivalent iff $\omega_\sigma(u,z)=\omega_\tau(u,z)$.
By Corollary~\ref{cor:nosing}, $\omega_\sigma(0,z)$ is analytic at $z=0$.
If $\omega_\sigma(0,z)=\omega_\tau(0,z)$, then the coefficients of the series expansions of these functions at $z=0$ coincide, so by equation~\eqref{eq:omeganol}, $d^\sigma_k=d^\tau_k$ for all $k\ge1$. But then
$\omega_\sigma(u,z)=\omega_\tau(u,z)$, so $\sigma$ and $\tau$ are strongly c-Wilf-equivalent.
\end{proof}

It is known~\cite{EliNoy} that there is one c-Wilf-equivalence class of non-overlapping patterns of length~$3$, represented by $132$, and two classes of non-overlapping patterns of length $4$, represented by $1342$ and $1243$.
For non-overlapping patterns of length $m\ge5$, we now show that the number of equivalence classes is at most the size of the set
$$\Delta_m=\{(a,b): 1\le a<b\le m-1,\ a+b\le m+1\}.$$

\begin{proposition}\label{prop:numbernol}
The number of (strong) c-Wilf-equivalence classes of non-overlapping patterns of length $m\ge 5$ is at most
$$\left\lfloor\frac{m^2-4}{4}\right\rfloor.$$
\end{proposition}

\begin{proof}
By Lemma~\ref{lem:strong}, c-Wilf and strong c-Wilf-equivalence classes coincide for non-overlapping patterns.
By Lemma~\ref{lem:1m}, the equivalence class of a pattern $\sigma\in\N_m$ is determined $\sigma_1$ and $\sigma_m$.
Since reversal and complementation preserve equivalence classes, each class has at least one pattern
with $\sigma_1<\sigma_m$ and $\sigma_1+\sigma_m\le m+1$. Additionally,
there is no non-overlapping pattern with $\sigma_1=1$ and $\sigma_m=m$, because such a pattern would start and end with an ascent, so it would have $m-2\in\O_\sigma$.
This shows that each class contains a pattern $\sigma$ with $(\sigma_1,\sigma_m)\in\Delta_m$, and thus the number of
classes is bounded from above by $|\Delta_m|$. By counting the number of pairs for each fixed value of $\sigma_1$, we get
$$|\Delta_m|=(m-2)+(m-3)+(m-5)+(m-7)+\dots=\begin{cases} \frac{m^2-4}{4} & \mbox{if $m$ is even},\\ \frac{m^2-5}{4} & \mbox{if $m$ is odd}.\end{cases}$$
\end{proof}

We conjecture that the formula in Proposition~\ref{prop:numbernol} is not just an upper bound, but the exact number of c-Wilf-equivalence classes of non-overlapping patterns of length $m\ge5$. In this direction, it is easy to show
that for every $(a,b)\in \Delta_m$ there is a pattern $\sigma\in\N_m$ with $\sigma_1=a$ and $\sigma_m=b$. Take, for example,
$$\sigma=a(a+1)(a+2)\dots\widehat{b}\dots(m-1)12\dots(a-1)mb,$$
where $\widehat{b}$ indicates that $b$ is missing.
This pattern is clearly non-overlapping if $a=1$. If $a\neq1$, then the only way for two occurrences of $\sigma$ to overlap is if the descent $mb$ of the first occurrence coincides with the descent $(m-1)1$ of the second occurrence.
However, this is impossible because the descent $mb$ forms a pattern $132$ with the entry preceding it, while the descent $(m-1)1$ forms a pattern $231$ with the entry preceding it.

To prove that the formula in Proposition~\ref{prop:numbernol} is exact, it remains to be shown that if
$\sigma,\tau\in\N_m$ are such that the pairs $(\sigma_1,\sigma_m)$ and $(\tau_1,\tau_m)$ are different and belong to $\Delta_m$, then $d^\sigma_k\neq d^\tau_k$ for some $k$.
This would imply that $\omega_\sigma(z)\neq \omega_{\tau}(z)$ and thus $\sigma$ and $\tau$ are not c-Wilf-equivalent.
Note, however, that there are examples such as $\sigma=23567184$ and $\tau=34671285$, which satisfy $d^\sigma_2=d^\tau_2$, even though $d^\sigma_3\neq d^\tau_3$. There are also pairs of longer patterns for which $d^\sigma_2>d^\tau_2$
but $d^\sigma_3<d^\tau_3$.

\ms

Finding the most and the least avoided non-overlapping patterns is closely related to finding the extremal values that $d^\sigma_2$ can take for $\sigma\in\N_m$. For this purpose, we take a closer look at the function $f(a,b)$.

\begin{lemma}\label{lem:fab}
\ben
\item[(i)] For $1\le a<b\le m-1$, \ $f(a,b)>f(a,b+1)$.
\item[(ii)] For $2\le a\le m/2$, \ $f(a-1,a)>f(a,a+1)$.
\item[(iii)] For $2\le a<b\le m$, \ $f(a,b)>f(a-1,b)$.
\een
\end{lemma}

\begin{proof}
The inequality in part (i), $$\binom{a+b-2}{a-1}\binom{2m-a-b}{m-b}>\binom{a+b-1}{a-1}\binom{2m-a-b-1}{m-b-1},$$  is equivalent to $\frac{2m-a-b}{m-b}>\frac{a+b-1}{b}$.
Subtracting $1$ from both sides, this is equivalent to $\frac{m-a}{m-b}>\frac{a-1}{b}$, which is clearly true because $a<b$.

The inequality in part (ii), $$\binom{2a-3}{a-2}\binom{2m-2a+1}{m-a}>\binom{2a-1}{a-1}\binom{2m-2a-1}{m-a-1},$$ is equivalent to $\frac{2m-2a+1}{m-a+1}>\frac{2a-1}{a}$.
Subtracting $2$ from both sides, this follows from the fact that $2a\le m$.
 Part (iii) follows from part~(i) using the symmetry $f(a,b)=f(m+1-b,m+1-a)$.
\end{proof}

Table~\ref{tab:d} illustrates the order relationships proved in Lemma~\ref{lem:fab} among the values $f(a,b)$ for $(a,b)\in \Delta_m$, with arrows pointing to the larger value in each pair.
We are interested in the two largest and the two smallest values of $f(a,b)$.

\begin{table}[htb]
$$\begin{array}{cccccccccccc}
f(1,2) &\leftarrow& f(1,3) &\leftarrow& f(1,4)& \leftarrow&\dots&\leftarrow & f(1,m-2) &\leftarrow & f(1,m-1) \\
 &\nwarrow& \downarrow& &\downarrow && & & \downarrow & & \downarrow \\
 && f(2,3) &\leftarrow& f(2,4)& \leftarrow&\dots&\leftarrow & f(2,m-2) &\leftarrow & f(2,m-1) \\
 && &\nwarrow& \downarrow & & &  & \downarrow &&  \\
 && && f(3,4) & \leftarrow & \dots&\leftarrow & f(3,m-2) &&\\
 && &&& \ddots & \dots& \Ddots &  &&\\
 \end{array}$$
\caption{\label{tab:d} Order relationships among the values $f(a,b)$ for $(a,b)\in \Delta_m$, where $m\ge4$.}
\end{table}

\begin{proposition}\label{prop:fab2}
For $m\ge5$ and $(a,b)\in \Delta_m\setminus\{(1,2),(2,3),(1,m-2),(1,m-1)\}$, we have
$$f(1,2)>f(2,3)>f(a,b)>f(1,m-2)>f(1,m-1).$$
The two largest and smallest values that $d^\sigma_2$ can take for $\sigma\in\N_m$ are given in Table~\ref{tab:d2}.
\end{proposition}

\begin{table}[htb]
\renewcommand{\arraystretch}{1.4}
\centering
\bt{c|c|c}
& value of $d^\sigma_2$ & \renewcommand{\arraystretch}{1} \bt{c} attained only when, up to reversal \\ and complementation, $\sigma$ satisfies\et \\
\hline
largest value & $\binom{2m-3}{m-2}$ & $\sigma_1=1$ and $\sigma_m=2$ \\
\hline
second largest value & $3\binom{2m-5}{m-3}$ & $\sigma_1=2$ and $\sigma_m=3$ \\
\hline
second smallest value & $\binom{m+1}{2}$ & $\sigma_1=1$ and $\sigma_m=m-2$\\
\hline
smallest value & $m$ & $\sigma_1=2$ and $\sigma_m=m-1$
\et
\caption{\label{tab:d2} The largest and smallest values of $d^\sigma_2$ for $\sigma\in\N_m$, where $m\ge5$.}
\end{table}

\begin{proof}
It is clear from Lemma~\ref{lem:fab} (see also Table~\ref{tab:d}) that the minimum and the two largest values of $f$ over $\Delta_m$ occur at the stated coordinates.
The reason that second smallest value is $f(1,m-2)$ rather than $f(2,m-1)$ is that $$f(1,m-2)=\binom{m+1}{2}<(m-1)^2=f(2,m-1)$$ for $m\ge5$.

The second statement follows using that, up to reversal and complementation, every $\sigma\in\N_m$ satisfies $(\sigma_1,\sigma_m)\in\Delta_m$, and $d^\sigma_2=f(\sigma_1,\sigma_m)$.
\end{proof}

It is now immediate that
\beq\label{eq:maxf} \underset{a+b\le m+1}{\max_{1\le a<b\le m}} f(a,b)=\max_{(a,b)\in\Delta_m\cup\{(1,m)\}} f(a,b)=f(1,2)=\binom{2m-3}{m-2},\eeq
which was used in the proof of Lemma~\ref{lem:2l}.

Using equation~\eqref{eq:snol}, Proposition~\ref{prop:boundsomega} can be reformulated as follows in the special case of non-overlapping patterns.

\begin{proposition}\label{prop:boundsomeganol}
For every $\sigma\in\N_m$ and $0<z\le C$,
$$1-z+\frac{z^m}{m!}-d^\sigma_2\frac{z^{2m-1}}{(2m-1)!}<\omega_{\sigma}(z)<  1-z+\frac{z^m}{m!}-d^\sigma_2\frac{z^{2m-1}}{(2m-1)!}+d^\sigma_3\frac{z^{3m-2}}{(3m-2)!}.$$
\end{proposition}

To warm up for the proof of Theorem~\ref{thm:nol}, which is the main result of this section, we use the above bounds to give a simpler proof of a theorem from~\cite{EliNoy2}, which is a special case of Theorem~\ref{thm:main}.

\begin{theorem}[\cite{EliNoy2}]\label{thm:monvsnonover}
For every $\sigma\in\N_m$, there exists $n_0$ such that
$$\alpha_n(\sigma)<\alpha_n(12\dots m)$$ for all $n\ge n_0$.
\end{theorem}

\begin{proof}
As in the proof of Theorem~\ref{thm:main}, it is enough to show that $\omega_{12\dots m}(z)<\omega_\sigma(z)$ for $0<z\le C$, because
this implies that the smallest positive zeroes of these two functions satisfy $\rho_{12\dots m}^{-1}<\rho_\sigma^{-1}$.

Combining Proposition~\ref{prop:ineqmon}, the lower bound in Proposition~\ref{prop:boundsomeganol}, and the fact that $d^\sigma_2\le\binom{2m-3}{m-2}$ by Proposition~\ref{prop:fab2}, it suffices to show that
$$\binom{2m-3}{m-2}\frac{z^{2m-1}}{(2m-1)!}<\frac{z^{m+1}}{(m+1)!}-\frac{z^{2m}}{(2m)!}$$
for $0<z\le C$. This in turn follows from
$$\frac{(m+1)m\,C^{m-2}}{(m-2)!(2m-1)(2m-2)}+\frac{C^{m-1}}{(2m)_{m-1}}<1,$$
which is easy to verify for $m\ge3$.
\end{proof}

\begin{theorem}\label{thm:nol}
For every $\sigma\in\N_m$, there exists $n_0$ such that $$\alpha_n(12\dots(m-2)m(m-1))\le\alpha_n(\sigma)\le\alpha_n(134\dots m2)$$ for all $n\ge n_0$.
\end{theorem}

\begin{proof}
Let $\tau=12\dots(m-2)m(m-1)$ and $\upsilon=134\dots m2$.
Assume without loss of generality that $(\sigma_1,\sigma_m)\in\Delta_m$, and that $\sigma$ is not c-Wilf-equivalent to either $\tau$ or $\upsilon$. Then,
by Lemma~\ref{lem:1m}, $(\sigma_1,\sigma_m)\notin\{(1,2),(1,m-1)\}$. Besides, $m\ge5$, since $\tau$ and $\sigma$ are the only non-overlapping patterns when $m=4$.
By Proposition~\ref{prop:fab2}, \beq\label{eq:boundsd2}\binom{m+1}{2}\le d^\sigma_2\le 3\binom{2m-5}{m-3}.\eeq

As in the proof of Theorem~\ref{thm:main}, it is enough to show that
\beq\label{eq:omegapq}\omega_{\upsilon}(z)<\omega_{\sigma}(z)<\omega_{\tau}(z)\eeq for $0<z<c$,
because this implies that the smallest positive zeroes of these functions satisfy $\rho_\upsilon^{-1}<\rho_\sigma^{-1}<\rho_\tau^{-1}$, and thus
there exists $n_0$ such that $\alpha_n(\tau)<\alpha_n(\sigma)<\alpha_n(\upsilon)$ for all $n\ge n_0$.

One can give simple formulas for all the coefficients of $\omega_{\upsilon}(z)$ and $\omega_{\tau}(z)$, although for this proof we only need
\beq\label{eq:dtu} d^\tau_2=m, \quad d^\upsilon_2=\binom{2m-3}{m-2} \quad \mbox{and} \quad d^\upsilon_3=d^\upsilon_2\binom{3m-4}{m-2},\eeq
which count linear extensions of the posets $D^\tau_2$, $D^\upsilon_2$ and $D^\upsilon_3$, drawn from left to right in Figure~\ref{fig:posetnolp}, respectively.

\begin{figure}[htb]
\centering
\begin{tikzpicture}[scale=.5]
\fill (0,0) circle (0.1) node[left]{$\pi_{1}$};
\fill (0,1) circle (0.1) node[left]{$\pi_{2}$};
\fill (0,2) circle (0.1) ;
\fill (0,3) circle (0.1) node[left]{$\pi_{m-2}$};
\fill (0,4) circle (0.1) node[left]{$\pi_{m}$};
\fill (0,5) circle (0.1) node[left]{$\pi_{m-1}$};
\fill (1,5) circle (0.1) node[right]{$\pi_{m+1}$};
\fill (1,6) circle (0.1) ;
\fill (1,7) circle (0.1) ;
\fill (1,8) circle (0.1) node[right]{$\pi_{2m-1}$};
\fill (1,9) circle (0.1) node[right]{$\pi_{2m-2}$};
\draw (0,0)--(0,5);
\draw[color=red] (0,4)--(1,5)--(1,9);
\end{tikzpicture}
\hspace{15mm}
\begin{tikzpicture}[scale=.5]
\fill (0,0) circle (0.1) node[left]{$\pi_{1}$};
\fill (0,1) circle (0.1) node[left]{$\pi_{m}$};
\fill (0,2) circle (0.1) node[left]{$\pi_{2}$};
\fill (0,3) circle (0.1) ;
\fill (0,4) circle (0.1) ;
\fill (0,5) circle (0.1) ;
\fill (0,6) circle (0.1) node[left]{$\pi_{m-2}$};
\fill (0,7) circle (0.1) node[left]{$\pi_{m-1}$};
\fill (1,2) circle (0.1) node[right]{$\pi_{2m-1}$};
\fill (1,3) circle (0.1) node[right]{$\pi_{m+1}$};
\fill (1,4) circle (0.1) ;
\fill (1,5) circle (0.1) ;
\fill (1,6) circle (0.1) ;
\fill (1,7) circle (0.1) node[right]{$\pi_{2m-3}$};
\fill (1,8) circle (0.1) node[right]{$\pi_{2m-2}$};
\draw (0,0)--(0,7);
\draw[color=red] (0,1)--(1,2)--(1,8);
\end{tikzpicture}
\hspace{15mm}
\begin{tikzpicture}[scale=.5]
\fill (0,0) circle (0.1) node[left]{$\pi_{1}$};
\fill (0,1) circle (0.1) node[left]{$\pi_{m}$};
\fill (0,2) circle (0.1) node[left]{$\pi_{2}$};
\fill (0,3) circle (0.1) ;
\fill (0,4) circle (0.1) ;
\fill (0,5) circle (0.1) ;
\fill (0,6) circle (0.1) node[left]{$\pi_{m-2}$};
\fill (0,7) circle (0.1) node[left]{$\pi_{m-1}$};
\fill (1,2) circle (0.1) node[right]{$\pi_{2m-1}$};
\fill (1,3) circle (0.1) node[right]{$\pi_{m+1}$};
\fill (1,4) circle (0.1) ;
\fill (1,5) circle (0.1) ;
\fill (1,6) circle (0.1) ;
\fill (1,7) circle (0.1) ;
\fill (1,8) circle (0.1) ;
\fill (2,4) circle (0.1) node[right]{$\pi_{3m-2}$};
\fill (2,5) circle (0.1) node[right]{$\pi_{2m}$};
\fill (2,6) circle (0.1) ;
\fill (2,7) circle (0.1) ;
\fill (2,8) circle (0.1) ;
\fill (2,9) circle (0.1) node[right]{$\pi_{3m-4}$};
\fill (2,10) circle (0.1) node[right]{$\pi_{3m-3}$};
\draw (0,0)--(0,7);
\draw[color=red] (0,1)--(1,2)--(1,8);
\draw[color=blue] (1,3)--(2,4)--(2,10);
\end{tikzpicture}
\caption{\label{fig:posetnolp} The cluster posets $D^\tau_2$ (left), $D^\upsilon_2$ (center), and $D^\upsilon_3$ (right).}
\end{figure}

By Proposition~\ref{prop:boundsomeganol} applied to $\omega_{\upsilon}(z)$ and $\omega_{\sigma}(z)$, the left inequality in~\eqref{eq:omegapq} will follow if we show that
$$d^\sigma_2\frac{z^{2m-1}}{(2m-1)!}<d^\upsilon_2\frac{z^{2m-1}}{(2m-1)!}-d^\upsilon_3\frac{z^{3m-2}}{(3m-2)!}$$
for $0<z<c$,
which is equivalent to
\beq\label{eq:dq}d^\sigma_2<d^\upsilon_2-d^\upsilon_3\frac{z^{m-1}}{(3m-2)_{m-1}}.\eeq 
Using~\eqref{eq:dtu} and the upper bound from~\eqref{eq:boundsd2}, the proof of inequality~\eqref{eq:dq} is reduced to showing that
$$\frac{3(m-1)}{2(2m-3)}+\frac{(2m-1)\,c^{m-1}}{3(3m-2)(m-1)!}<1$$
for $m\ge5$, which is straightforward, since the left hand side is clearly decreasing in $m$.

For the right inequality in~\eqref{eq:omegapq}, by Proposition~\ref{prop:boundsomeganol} applied to $\omega_{\sigma}(z)$ and $\omega_{\tau}(z)$, it suffices to show that
\beq\label{eq:pd} d^\tau_2<d^\sigma_2-d^\sigma_3\frac{z^{m-1}}{(3m-2)_{m-1}}\eeq
for $0<z<c$. A simple upper bound on $d^\sigma_3$ that follows from equations~\eqref{eq:n+l} and~\eqref{eq:boundsd2}
is $$d^\sigma_3\le d^\sigma_2\binom{3m-2}{m-1}\le 3\binom{2m-5}{m-3}\binom{3m-2}{m-1}.$$
Combining this with the the lower bound from~\eqref{eq:boundsd2} and using~\eqref{eq:dtu},
inequality~\eqref{eq:pd} follows from the fact that
$$\frac{3(2m-5)_{m-3}\,c^{m-1}}{(m-3)!(m+1)!}+\frac{1}{m+1}<\frac{1}{2}$$
for $m\ge5$, which is again straightforward because the left hand side is decreasing in $m$.
\end{proof}

\section{The least avoided pattern}\label{sec:least}

In this section we prove that the pattern $12\dots(m-2)m(m-1)$ is the least avoided not only among non-overlapping patterns, but also among all patterns of length $m$.
This settles a conjecture of Nakamura~\cite[Conjecture~2]{Nak}.

\begin{theorem}\label{thm:nak}
For every $\sigma\in\S_m$, there exists $n_0$ such that $$\alpha_n(12\dots(m-2)m(m-1))\le\alpha_n(\sigma)$$ for all $n\ge n_0$.
\end{theorem}

\begin{proof}
Let $\tau=12\dots(m-2)m(m-1)$, and let $p=\min\O_\sigma$.
If $\sigma$ is monotone, the result follows from Theorem~\ref{thm:main}, and if $\sigma$ is non-overlapping, it was proved in Theorem~\ref{thm:nol}. Thus,
we assume that $\sigma$ is neither monotone nor non-overlapping. These conditions imply that $p\neq1$ and $p\neq m-1$, respectively, so we have $2\le p\le m-2$ and in particular $m\ge4$.

Again, as in the proof Theorem~\ref{thm:main}, it suffices to show that
\beq\label{eq:omegatau}\omega_{\sigma}(z)<\omega_{\tau}(z)\eeq for $0<z<c$,
which implies that that $\rho_\tau<\rho_\sigma$. In the rest of this proof we assume that $0<z<c$.
Using the first upper bound in Proposition~\ref{prop:boundsomega}, the lower bound in Proposition~\ref{prop:boundsomeganol} for the pattern $\tau$, and the fact that $d^\tau_2=m$, equation~\eqref{eq:omegatau} will follow if we show that
\beq\label{eq:s23}\frac{m\,z^{2m-1}}{(2m-1)!}<s^\sigma_2(z)-s^\sigma_3(z).\eeq
From equation~\eqref{eq:skcompare} with $k=2$, we get
$$s^\sigma_3(z)\le s^\sigma_2(z)\sum_{\ell\in\O_\sigma} \frac{c^\ell}{\ell!},$$
and so
$$s^\sigma_2(z)-s^\sigma_3(z)\ge s^\sigma_2(z)\left(1-\sum_{\ell\in\O_\sigma} \frac{c^\ell}{\ell!}\right)\ge s^\sigma_2(z)\left(1-\sum_{\ell=p}^{\infty}\frac{c^\ell}{\ell!}\right).$$
Combining this with the simple lower bound $$s^\sigma_2(z)=\sum_{\ell\in\O_\sigma} r^\sigma_{m+\ell,2}\frac{z^{m+\ell}}{(m+\ell)!}\ge \frac{z^{m+p}}{(m+p)!},$$
the proof of equation~\eqref{eq:s23} it reduced to showing that
$$\frac{m\,z^{2m-1}}{(2m-1)!} < \frac{z^{m+p}}{(m+p)!}\left(1-\sum_{\ell=p}^{\infty}\frac{c^\ell}{\ell!}\right),$$ which is in turn a consequence of
\beq\label{eq:pm} \frac{m\,c^{m-p-1}}{(2m-1)_{m-p-1}}+\sum_{\ell=p}^{\infty}\frac{c^\ell}{\ell!}<1.\eeq

Let us prove inequality~\eqref{eq:pm} for $m\ge5$ and $2\le p\le m-2$. Denote its left hand side by $L(m,p)$.
For fixed $p$, $L(m,p)$ is decreasing in $m$, since for the summand that depends on $m$, the quotient of its evaluation at $m+1$ by its evaluation at $m$ is
$$\frac{(m+1)(m+p+1)c}{m(2m+1)(2m)}<\frac{(m+1)c}{m(2m+1)}<1.$$
Thus, for each fixed $p\ge3$, the maximum over $m\ge p+2$ of $L(m,p)$ is attained when $m=p+2$, and it equals
$$\frac{(p+2)\,c}{2p+3}+\sum_{\ell=p}^{\infty}\frac{c^\ell}{\ell!}.$$
This function is decreasing in $p$, so it is bounded from above by its value when $p=3$, which is
$5c/9+e^c-1-c-c^2/2<1$.
For $p=2$ and $m\ge5$, the maximum of $L(m,p)$ is attained when $m=5$, and it equals
$5c^2/72+e^c-1-c<1$.

The only case that remains to be proved is when $m=4$ and $p=2$. In this case, inequality~\eqref{eq:pm} does not hold, but we can check equation~\eqref{eq:s23} directly.
Up to reversal and complementation, the patterns with $m=4$ and $p=2$ are $2413$, $2143$, $1324$ and $1423$. One can easily compute
$$s^\sigma_2(z)=r^\sigma_{6,2}\frac{z^6}{6!}+r^\sigma_{7,2}\frac{z^7}{7!} \quad \mbox{and} \quad s^\sigma_3(z)=r^\sigma_{8,3}\frac{z^8}{8!}+r^\sigma_{9,3}\frac{z^9}{9!}+r^\sigma_{10,3}\frac{z^{10}}{10!}$$
for each one of these patterns (see Table~\ref{tab:r23} and Figure~\ref{fig:posetsm4p2}), and verify that inequality~\eqref{eq:s23} holds for each one of them.
Alternatively, \eqref{eq:s23} can be proved for these patterns by computing only $s^\sigma_2(z)$ and using the first inequality in~\eqref{eq:n+l} to deduce that $s^\sigma_3(z)\le s^\sigma_2(z)(7c^2/24+7c^3/60)$.
\end{proof}

\begin{table}[htb]
$$\begin{array}{|c||c|c||c|c|c|}
\hline
\sigma & r^\sigma_{6,2} & r^\sigma_{7,2} & r^\sigma_{8,3} & r^\sigma_{9,3}& r^\sigma_{10,3}\\ \hline \hline
2413 & 2 & 9 & 5 & 108 & 234\\ \hline
2143 & 1 & 9 & 1 & 30 & 234\\ \hline
1324 & 2 & 1 & 5 & 4 & 1\\ \hline
1423 & 1 & 4 & 1 & 16 & 28\\ \hline
\end{array}$$
\caption{\label{tab:r23} The cluster numbers with $k=2,3$ for the patterns with $m=4$ and $p=2$.}
\end{table}

\begin{figure}[htb]
\centering
\bt{c|c|c|c}
\begin{tikzpicture}[scale=.5]
\fill (0,0) circle (0.1) node[left]{$\pi_3$};
\fill (0,1) circle (0.1) node[left]{$\pi_1$};
\fill (0,2) circle (0.1) node[left]{$\pi_4$};
\fill (0,3) circle (0.1) node[left]{$\pi_2$};
\fill (1,-1) circle (0.1) node[right]{$\pi_5$};
\fill (1,1) circle (0.1) node[right]{$\pi_6$};
\draw (0,0)--(0,3);
\draw[red] (1,-1)--(0,0)--(1,1)--(0,2);
\end{tikzpicture}
\begin{tikzpicture}[scale=.5]
\fill (0,0) circle (0.1) node[left]{$\pi_3$};
\fill (0,1) circle (0.1) node[left]{$\pi_1$};
\fill (0,2) circle (0.1) node[left]{$\pi_4$};
\fill (0,3) circle (0.1) node[left]{$\pi_2$};
\fill (1,1) circle (0.1) node[right]{$\pi_6$};
\fill (1,3) circle (0.1) node[right]{$\pi_7$};
\fill (1,4) circle (0.1) node[right]{$\pi_5$};
\draw (0,0)--(0,3);
\draw[red] (1,1)--(0,2)--(1,3)--(1,4);
\end{tikzpicture}
&
\begin{tikzpicture}[scale=.5]
\fill (0,0) circle (0.1) node[left]{$\pi_2$};
\fill (0,1) circle (0.1) node[left]{$\pi_1$};
\fill (0,2) circle (0.1) node[left]{$\pi_4$};
\fill (0,3) circle (0.1) node[left]{$\pi_3$};
\fill (0,4) circle (0.1) node[right]{$\pi_6$};
\fill (0,5) circle (0.1) node[right]{$\pi_5$};
\draw (0,0)--(0,3);
\draw[red] (0,3)--(0,5);
\end{tikzpicture}
\begin{tikzpicture}[scale=.5]
\fill (0,0) circle (0.1) node[left]{$\pi_2$};
\fill (0,1) circle (0.1) node[left]{$\pi_1$};
\fill (0,2) circle (0.1) node[left]{$\pi_4$};
\fill (0,3) circle (0.1) node[left]{$\pi_3$};
\fill (1,1) circle (0.1) node[right]{$\pi_5$};
\fill (1,3) circle (0.1) node[right]{$\pi_7$};
\fill (1,4) circle (0.1) node[right]{$\pi_6$};
\draw (0,0)--(0,3);
\draw[red] (1,1)--(0,2)--(1,3)--(1,4);
\end{tikzpicture}
&
\begin{tikzpicture}[scale=.5]
\fill (0,0) circle (0.1) node[left]{$\pi_1$};
\fill (0,1) circle (0.1) node[left]{$\pi_3$};
\fill (-1,2) circle (0.1) node[left]{$\pi_2$};
\fill (0,3) circle (0.1) node[left]{$\pi_4$};
\fill (1,2) circle (0.1) node[right]{$\pi_5$};
\fill (0,4) circle (0.1) node[right]{$\pi_6$};
\draw (0,0)--(0,1)--(-1,2)--(0,3);
\draw[red] (0,1)--(1,2)--(0,3)--(0,4);
\end{tikzpicture}
\begin{tikzpicture}[scale=.5]
\fill (0,0) circle (0.1) node[left]{$\pi_1$};
\fill (0,1) circle (0.1) node[left]{$\pi_3$};
\fill (0,2) circle (0.1) node[left]{$\pi_2$};
\fill (0,3) circle (0.1) node[left]{$\pi_4$};
\fill (0,4) circle (0.1) node[right]{$\pi_6$};
\fill (0,5) circle (0.1) node[right]{$\pi_5$};
\fill (0,6) circle (0.1) node[right]{$\pi_7$};
\draw (0,0)--(0,3);
\draw[red] (0,3)--(0,6);
\end{tikzpicture}
&
\begin{tikzpicture}[scale=.5]
\fill (0,0) circle (0.1) node[left]{$\pi_1$};
\fill (0,1) circle (0.1) node[left]{$\pi_3$};
\fill (0,2) circle (0.1) node[right]{$\pi_5$};
\fill (0,3) circle (0.1) node[right]{$\pi_6$};
\fill (0,4) circle (0.1) node[left]{$\pi_4$};
\fill (0,5) circle (0.1) node[left]{$\pi_2$};
\draw (0,0)--(0,1.5);
\draw (0,3.5)--(0,5);
\draw[red] (0,1.5)--(0,3.5);
\end{tikzpicture}
\begin{tikzpicture}[scale=.5]
\fill (0,0) circle (0.1) node[left]{$\pi_1$};
\fill (0,1) circle (0.1) node[left]{$\pi_3$};
\fill (0,2) circle (0.1) node[left]{$\pi_4$};
\fill (0,3) circle (0.1) node[left]{$\pi_2$};
\fill (1,3) circle (0.1) node[right]{$\pi_6$};
\fill (1,4) circle (0.1) node[right]{$\pi_7$};
\fill (1,5) circle (0.1) node[right]{$\pi_5$};
\draw (0,0)--(0,3);
\draw[red] (0,2)--(1,3)--(1,5);
\end{tikzpicture}
\\
$\sigma=2413$ & $\sigma=2143$ & $\sigma=1324$ & $\sigma=1423$
\et
\caption{\label{fig:posetsm4p2} Cluster posets for $k=2$. For each $\sigma$, their linear extensions are counted by $r^\sigma_{6,2}$ (left) and $r^\sigma_{7,2}$ (right).}
\end{figure}
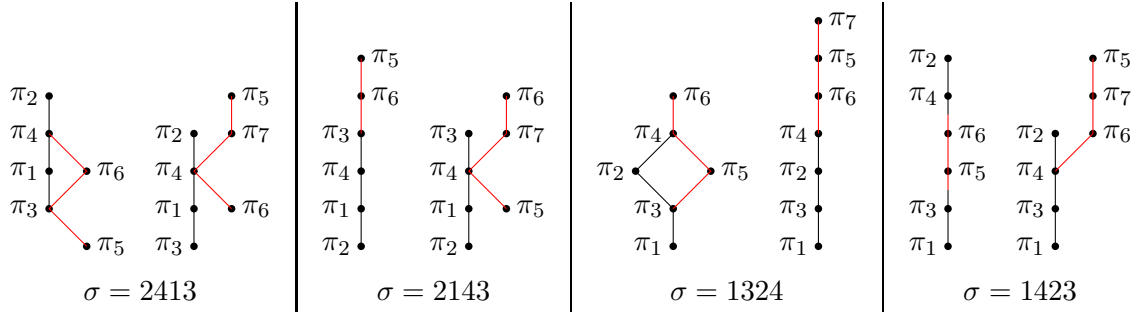

\section{Final remarks}\label{sec:final}

It is wide open to find combinatorial proofs of Theorems~\ref{thm:main}, \ref{thm:nol} and~\ref{thm:nak}. A combinatorial proof of Theorem~\ref{thm:main}, for example, could
be a length-preserving injection from permutations avoiding $\sigma$ to permutations avoiding $12\dots m$, for arbitrary $\sigma\in\S_m$.
This problem is solved only for $m=3$: a combinatorial proof of the inequality $\alpha_n(132)<\alpha_n(123)$ is given in~\cite{EliNoy}.
\ms

A different open problem, which is mentioned in~\cite{EliNoy2}, is to find a proof of Theorem~\ref{thm:EKP} that does not rely on spectral theory.
By Theorem~\ref{thm:sing}, the singularities of $P_\sigma(0,z)$ nearest to the origin have modulus $\rho_\sigma^{-1}$, and by Corollary~\ref{cor:nosing}, they are zeroes of $\omega_\sigma(z)$.
If one can show that $z=\rho_\sigma^{-1}$ is the only zero of $\omega_\sigma(z)$ with $|z|=\rho_\sigma^{-1}$, and that this zero is simple,
then Theorem~\ref{thm:EKP} will follow from standard singularity analysis~\cite[Theorem IV.10]{FS}.
While we have not been able to prove in general that $\omega_\sigma(z)$ has no other zeroes of minimum modulus,
the following result, proved along the lines of Propositions~\ref{prop:decreasing} and~\ref{prop:boundsomega}, shows that
$z=\rho_\sigma^{-1}$ is {\em simple} zero of $\omega_\sigma(z)$.

\begin{proposition}
For every $\sigma\in\S_m$, $$\omega_\sigma'(\rho_\sigma^{-1})<0.$$
\end{proposition}

\begin{proof}
We assume that $m\ge4$, since for $m=3$ the result follows immediately from the explicit expressions for $\omega_\sigma(z)$ given in~\cite{EliNoy}. With this assumption, we have $\rho_\sigma^{-1}< c$ by Corollary~\ref{cor:c1}.
Differentiating equation~\eqref{eq:omegaalt}, we get
$$\omega'_\sigma(z)=-1-\sum_{k\ge1}s'_k(z)(-1)^k,$$
where $s_k'(z)$ denotes the derivative of $s^\sigma_k(z)$.

We claim that for any $0<z<c$, the sequence $\{s'_k(z)\}_{k\ge1}$ is decreasing. The proposition follows from this claim because then
$$\omega_\sigma'(\rho_\sigma^{-1})<-1+s'_1(\rho_\sigma^{-1})<-1+\frac{c^{m-1}}{(m-1)!}<0.$$

To prove the claim, we use the same notation as in the proof of Proposition~\ref{prop:decreasing}. Differentiating equation~\eqref{eq:sk} we have
\beq\label{eq:s'k}s'_k(z)=\sum_{(i_1,i_2,\dots,i_k)\in\I^\sigma_{k}}\lext(Q^\sigma_{i_1,i_2,\dots,i_k})\frac{z^{i_k+m-2}}{(i_k+m-2)!},\eeq
and since $n-m+2\ell\le n+\ell-1$, the first inequality in~\eqref{eq:n+l} gives
$$\lext(Q^\sigma_{i_1,\dots,i_k,i_k+\ell})
\le\binom{n+\ell-1}{\ell}\lext(Q^\sigma_{i_1,\dots,i_k}),$$
so
$$\lext(Q^\sigma_{i_1,\dots,i_k,i_k+\ell})\frac{z^{n+\ell-1}}{(n+\ell-1)!}\le\lext(Q^\sigma_{i_1,\dots,i_k})\frac{z^{n-1}}{(n-1)!} \frac{z^\ell}{\ell!}.$$
Summing over $\ell\in\O_\sigma$ and over $(i_1,\dots,i_k)\in\I^\sigma_k$, and using~\eqref{eq:s'k}, we get
\beq\label{eq:s'kcompare} s'_{k+1}(z)\le s'_k(z) \sum_{\ell\in\O_\sigma} \frac{z^\ell}{\ell!}< s'_k(z) \sum_{\ell\in\O_\sigma} \frac{c^\ell}{\ell!}\eeq
for $0<z<c$. If $\sigma$ is not monotone, then $1\notin\O_\sigma$, and
$$\sum_{\ell\in\O_\sigma} \frac{c^\ell}{\ell!}\le \sum_{\ell\ge2} \frac{c^\ell}{\ell!}=e^{c}-1-c<1,$$
so $s'_{k+1}(z)< s'_k(z)$ and we are done.

If $\sigma$ is monotone, then an argument analogous to the last part of the proof of Proposition~\ref{prop:decreasing} shows that, for $0<z<c$,
$$s'_{k+1}(z)\le s'_k(z) \sum_{\ell=1}^{m-1} \frac{z^\ell}{(m+\ell-1)_\ell}<s'_k(z) \frac{c/m}{1-c/m}<s'_k(z).$$
\end{proof}

\end{document}